\documentclass[12pt]{article}
\usepackage{amsmath, amssymb,amsthm}%

\newtheorem{theorem}{Theorem}

\theoremstyle{definition}
\newtheorem*{defen}{Definition}
\newtheorem*{zam}{Remarks}
\newtheorem*{remark}{Remark}

\newcommand{\W}{W^{\mathrm{reg}}}

\newcommand{\vol}{\mathop{\mathrm{vol}}}

\begin{document}

\vspace{10pt}

\begin{center}
\large \textsc{Angle measures of some cones associated with
finite reflection groups}

\vspace{5pt}

P.\,V.\,Bibikov\footnote{The first author was partially supported by the Moebius Contest
Foundation for Young
Scientists}, V.\,S.\,Zhgoon\footnote{The second author was partially supported by the following grants RFBR 09-01-00287, 09-01-12169.}
\end{center}

\vspace{10pt}

Let $V$ be a Euclidean vector space of dimension $n$ with the
inner product $(\cdot,\cdot)$.
For a convex polyhedral cone $C$ we denote by $C^\circ$  its open kernel,
by $\langle C \rangle$ its linear span,
and by  $C^*$  its \emph{dual cone}, i.e. $C^*=\{v\in
V:\, (v,u)\geqslant 0 \;\forall \, u\in C\}$.
Denote by $\sigma(C)$ \emph{the relative angle
measure of $C$}, i.e. $\sigma(C)=\vol(C\cap B)/\vol (B)$, where
$B\subset \langle C \rangle$ is the unit ball centered at the
origin.
Let $F$ be a $k$-dimensional face of  some solid cone  $C$.
By $C/F$ we denote the orthogonal projection of  $C$  to the  subspace $\langle F \rangle^\bot$.
We note that $(C/F)^*$ is  an $(n-k)$-dimensional
face of  $C^*$ that is orthogonal to $F$.

Let now $W\subset O(V)$ be a  finite reflection group.
 Set $W^k=\{w\in W:\dim\ker(1-w)=k\}$ and denote $\W:=W^0$. For a subspace $U\subset V$ we denote by $W_U$ the subgroup of $W$
 that fixes $U$ pointwise.

 The aim of this paper is to prove the following theorem
conjectured by the first author in~\cite{BiUMN}.

\begin{theorem}\label{th.obch.} For a fundamental cone $C$ of a finite reflection group $W$ and for each $k=0,\ldots,n$ we have
$$\sum\limits_{{F\subset C,\ \dim
F=k}}\sigma(F)\cdot\sigma((C/F)^*)=|W^{k}|/|W|,$$
where $F$ runs over the $k$-dimensional faces of $C$.
\end{theorem}

The following two results are crucial for the proof of Theorem \ref{th.obch.}.
The first one is the
  fundamental result of Waldspurger~\cite{Wald}, for the simplest proof of which  we
refer the reader to~\cite{BiZhUMN,
BiZhIAN}.
 The second one is  the so-called ``Curious
Identity'' of De Concini and Procesi \cite{ConPro} (see also \cite{ BiUMN, Den}).

\begin{theorem}[Walsdpurger]\label{W}
 $C^*=\bigsqcup\limits_{w\in W}(1-w)C^\circ$.
\end{theorem}

\begin{theorem}[Curious Identity]\label{DC}
$\sigma(C^*)=|\W|/|W|.$
\end{theorem}
\begin{proof}[Proof of Theorem 3]
 Theorem \ref{W} implies that
$wC^*=\bigcup_{w'\in \W}(1-w')wC.$
 Thus we obtain
$$
|W|\sigma(C^*)=\sum_{w\in W}\sigma(wC^*)=\sum_{ w'\in
\W}\sum_{w\in W}\sigma((1-w')wC)=\sum_{w'\in \W}\sigma((1-w')V)
=|\W|.
$$
\end{proof}

\begin{remark} It is easy to see from the previous proof that the number of cones $wC^*$ covering a
generic point of $V$ is equal to $|\W|$.
\end{remark}

\begin{proof}[Proof of Theorem~\ref{th.obch.}]
Consider the sum $S:=\sum_{(C,F)}\sigma(F)\cdot\sigma((C/F)^*)$
over all  pairs $(C,F)$, where $C$ is a fundamental cone of $W$ and
$F$ is a $k$-dimensional face of $C$.
 This sum is equal to
the left hand side of the required formula multiplied by $|W|$. We shall calculate $S$ in a different way.

 Let us recall that
any two fundamental cones of $W$ with a common face $F$ are conjugate
by a unique element of the reflection subgroup $W_F$ that fixes $F$
pointwise. We also note that $C/F$ is a fundamental cone for the action of $W_F$  on $\langle F \rangle^\bot$.
By Theorem \ref{DC} we have $\sigma((C/F)^*)=|\W_F|/|W_F|$. If we take the sum of $\sigma((C/F)^*)$
over all $C$ that contain a fixed face $F$, we get $|\W_F|$.    The group $W_F$ and the measure
 $\sigma((C/F)^*)$ depend only on the subspace $U:=\langle F \rangle$ that is an intersection of reflection
hyperplanes.
Let us take the sum of $\sigma(F)\cdot\sigma((C/F)^*)$ over all pairs $(C,F)$ such that $F\subset U$
for a fixed $k$-dimensional space $U$. We get $|\W_U|$ multiplied by the total measure of faces $F\subset U$,
which is equal to one, since the faces $F$ decompose $U$. Taking the sum over all  subspaces $U$ we obtain
$S=|W^k|$.
\end{proof}

\begin{zam}
1. Given a cone $C$ and its face $F$,
 we define the cone $F\oplus (C/F)^*$.
It follows from the previous proof  that for  any $k$-dimensional
subspace $U$,
 which is an intersection of reflection hyperplanes,
  a generic point of $V$ is covered by $|\W_U|$ cones  $F\oplus (C/F)^*$ with $F\subset U$.

2. The sums from  Theorem~\ref{th.obch.} can be expressed
in terms of the exponents $m_1,\ldots,m_n$ of $W$ with a help of the Solomon formula \cite{Sol}:
$\sum\limits_{k=0}^n|W^{n-k}| t^k=\prod\limits_{k=0}^n(1+m_kt)$.
\end{zam}

\begin{defen} Let $C$ be a fundamental cone of $W$.
We say that two $k$-dimensional faces
$F$ and $F'$ of $C$ are equivalent if there exists
an element $w\in W$ such that $w\langle F\rangle=\langle
F'\rangle$.
\end{defen}

Denote by $N_F$ the subgroup of $W$ normalizing
 $\langle F\rangle$.
 Consider the sum $(|W|/|W_{F}|)\cdot\sum_{F'\sim
F}\sigma(F')$ of relative angle measures for the $W$-translates
 of the $k$-dimensional faces $F'\subset C$  that are equivalent
to $F$. It is the same as the total measure of the $W$-translates of $F'\subset\langle F\rangle$.
 The latter is equal to $|W|/|N_F|$ multiplied by the total measure of the faces $F'\subset \langle F\rangle$
 which is equal to $1$. Thus we get $\sum_{F'\sim
F}\sigma(F')=|W_F|/|N_F|$.
 Using Theorems \ref{th.obch.}, \ref{DC} and applying
the previous equality we get:
$$\sum_{F}|\W_F|/|N_F|=\sum_F\Big(\sum\limits_{F'\sim
F}\sigma(F')\Big)\cdot\sigma((C/F)^*)=|W^{k}|/|W|,$$
where the first  sum is taken over representatives of all equivalence classes of  $k$-dimensional faces $F \subset C$.

The authors are grateful to E.\,B.\,Vinberg  for valuable discussions.

\end{document}